\documentclass[12pt, a4paper]{article}
\usepackage{amsmath, amsthm, amsfonts, amscd, amssymb, epsfig, setspace, url, graphics}
\usepackage[all]{xy}

\newtheorem{theorem}{Theorem}[section]
\newtheorem{proposition}[theorem]{Proposition}
\newtheorem{lemma}[theorem]{Lemma}
\newtheorem{corollary}[theorem]{Corollary}

\newtheorem{definition}[theorem]{Definition}

\newtheorem{remark}[theorem]{Remark}

\newcommand{\Z}{{\mathbb Z}}
\newcommand{\Q}{{\mathbb Q}}
\newcommand{\R}{{\mathbb R}}
\newcommand{\C}{{\mathbb C}}

\newcommand{\PP}{{\mathbb P}}
\newcommand{\cO}{{\mathcal O}}

\DeclareMathOperator{\rank}{{rank}}
\DeclareMathOperator{\Sec}{{Sec}}
\DeclareMathOperator{\Tan}{{Tan}}

\title{K-theory, LQEL manifolds and Severi varieties}
\author{Oliver Nash}
\date{August 2013}

\begin{document}
\maketitle

\begin{abstract}
  We use topological K-theory to study non-singular varieties with
  quadratic entry locus. We thus obtain a new proof of Russo's Divisibility Property
  for locally quadratic entry locus manifolds. In particular we obtain a K-theoretic proof
  of Zak's theorem that the dimension of a Severi variety must be 2, 4, 8 or 16
  and so resolve a conjecture of Atiyah and Berndt.
  We also show how the same methods applied to dual varieties recover the
  Landman parity theorem.
\end{abstract}

\section{Introduction}\label{Intro_sect}
  Zak's celebrated classification of Severi varieties \cite{MR1234494} establishes
  that there are only four such varieties and that they correspond to projective planes
  over the four division algebras. Taking into account the classical results relating
  K-theory, division algebras and projective planes, Atiyah and Berndt \cite{MR2039984}
  conjectured that there should be a K-theoretic proof that the dimension of a
  Severi variety was necessarily 2, 4, 8 or 16.

  By taking up an old approach of Fujita and Roberts \cite{MR630774} and
  Tango \cite{MR670136} but replacing characteristic classes with K-theory, we
  are able to provide the conjectured K-theoretic proof of the Severi variety dimension
  restriction. In fact our results sit naturally in the domain of Russo's LQEL
  manifolds \cite{MR2501303} and we provide a new K-theoretic proof of his
  Divisibility Property for LQEL manifolds.

  The method we employ is to consider the K-theoretic consequences of the existence of
  the generalized Euler sequence associated to a vector bundle.
  The generalized Euler sequence of a vector bundle $V$ over a base $B$
  is the natural exact sequence on the total space of the projectivization $\PP(V)$:
  \begin{align}\label{gen_Euler}
    0 \to \cO \to p^*V(1) \to T\PP(V) \to p^*TB \to 0
  \end{align}
  where $p : \PP(V) \to B$ is the bundle map, $p^*V(1) = p^*V \otimes \cO(1)$ and
  $\cO(1)$ is the dual of the tautological line bundle on $\PP(V)$. In the special
  case $B$ is a point this is the familiar Euler sequence on projective space
  (see e.g., \cite{MR0463157} II.8.13) and in the general case as above, it
  essentially reduces to this since $\PP(V) \to B$ is locally trivial.

  We obtain our results by taking $V$ to be the (extended) tangent bundle
  of a projective variety and noting that $\PP(V)$ also fibres over the secant variety.
  In the case that the variety is an LQEL manifold, the irreducible components of
  a general fibre of the map to the secant variety are non-singular quadrics.
  As a result, the topological
  K-theory of such a quadric carries a special relation in K-theory
  which turns out to be very restrictive.

  The problem of classifying Severi varieties was first posed by Hartshorne
  in his influential paper \cite{MR0384816} and is closely related to his
  complete intersection conjecture, op. cit. Since Hartshorne's motivation for this
  conjecture was partly topological (specifically, the Barth-Larsen theorems)
  it is tempting to wonder, in view of the results here and of Ionescu and Russo's recent
  proof \cite{MR3038714} of the complete intersection conjecture for quadratic manifolds,
  what relevance topological K-theory may have for the complete
  intersection conjecture.

\section{LQEL manifolds}\label{LQEL_section}
  We recall the basic definitions for the reader's convenience and to fix notation
  and terminology. For examples, further details and proofs of the assertions below we recommend
  Russo \cite{MR2028046,MR2501303}, Fujita and Roberts \cite{MR630774}
  and of course Zak's excellent foundational monograph \cite{MR1234494}. Our
  definitions are slightly simpler because we stick to non-singular varieties. We
  work over $\C$ throughout as we will obtain our results by
  using topological K-theory.

  \begin{definition}
    Let $Y \subseteq \PP^N$ be a non-singular irreducible projective variety with secant variety
    $\Sec(Y) \subseteq \PP^N$ and $z \in \Sec(Y) - Y$. The entry locus
    of $Y$ with respect to $z$ is defined to be:
    \begin{align*}
      \Sigma_z(Y) = \{y \in Y\mid \mbox{the line $yz$ is a tangent or secant of $Y$}\}
    \end{align*}
  \end{definition}

  The general entry locus is a projective variety
  with pure dimension equal to the secant deficiency:
  \begin{align*}
    \dim \Sigma_z(Y) = \delta = 2n+1 - \dim \Sec(Y)
  \end{align*}

  \begin{definition}
    Let $Y \subseteq \PP^N$ be a non-singular irreducible projective variety.
    Following Russo \cite{MR2501303}
    we say $Y$ is a locally quadratic entry locus (LQEL) manifold of type $\delta$ if
    each irreducible component of a general entry locus is a (non-singular,
    $\delta$-dimensional) quadric.
  \end{definition}

  \begin{definition}
    Let $Y \subseteq \PP^N$ be a non-singular irreducible projective variety with tangent
    variety $\Tan(Y) \subseteq \PP^N$ and $z \in \Tan(Y) - Y$. The tangent locus
    of $Y$ with respect to $z$ is defined to be:
    \begin{align*}
      \tau_z(Y) = \{y \in Y\mid z \in \mathbb{T}_yY \}
    \end{align*}
    where $\mathbb{T}_yY \subseteq \PP^N$ is the embedded tangent space of $Y$
    at $y$.
  \end{definition}
  The general tangent locus is a projective variety with
  pure dimension equal to the tangent deficiency:
  \begin{align*}
    \dim \tau_z(Y) = \delta_\tau = 2n - \dim\Tan(Y)
  \end{align*}

  We recall Zak's theorem that $\delta > 0$ iff $\Tan(Y) = \Sec(Y)$ so that
  in this case we have $\delta_\tau = \delta - 1$.

  \begin{lemma}
    Let $Y \subseteq \PP^N$ be an LQEL manifold of type $\delta > 0$ and $z \in \Sec(Y) - Y$
    a general point. For each irreducible component $Q$ of the entry locus $\Sigma_z(Y)$, the
    polar of $z$ with respect to $Q$ determines a non-singular hyperplane section $F$ of $Q$.
    These non-singular $(\delta-1)$-dimensional quadrics $F$ are the
    irreducible components of the tangent locus $\tau_z(Y)$.
  \end{lemma}
  \begin{proof}[Proof:]
    Let $F$ be an irreducible component of $\tau_z(Y)$. Since $\tau_z(Y) \subset \Sigma_z(Y)$
    we must have $F \subset Q$ for some irreducible component $Q$ of $\Sigma_z(Y)$. Since
    any tangent line of $Y$ passing through $z$ can be obtained as a limit of secants of passing
    through $z$ we have:
    \begin{align*}
      F &= \{y \in Q\mid z \in \mathbb{T}_yY\}\\
        &= \{y \in Q\mid z \in \mathbb{T}_yQ\}\\
        &= Q \cap H_z
    \end{align*}
    where $H_z = \{x \in M\mid q(x, z) = 0\}$ is the polar of $z$ with respect to $Q$,
    $M \subseteq \PP^N$ is the $(\delta+1)$-dimensional linear span of $Q$ and $q$ is the
    quadratic form on $M$ cutting out $Q$.
  \end{proof}

  Our key observation is that an irreducible component of a general tangent locus supports
  some rather special topology as a result of the ambient LQEL geometry.
  \begin{proposition}\label{LQEL_k-thy_divisibility}
    Let $Y \subseteq \PP^N$ be an $n$-dimensional LQEL manifold of type
    $\delta > 0$ and let $F \subseteq Y$ be an irreducible component of
    a general tangent locus. Then:
    \begin{align*}
      \mbox{$1 + \cO(1)$ divides $2(n-\delta)$ in $K(F)$}
    \end{align*}
    where $K(F)$ is the topological (complex) K-theory of $F$ (with its analytic
    topology) and $\cO(1)$ is the class in $K(F)$ represented by the restriction of
    the hyperplane section bundle to $F$.
  \end{proposition}
  \begin{proof}[Proof:]
    We take up the ideas of \cite{MR630774} and \cite{MR670136} except that
    instead of computing Chern classes, we derive a relation in K-theory.
    Thus let:
    \begin{align*}
      \Theta = \{(y, z) \in Y \times \Sec(Y) \mid z \in \mathbb{T}_yY\}
    \end{align*}
    Recall that the embedded tangent space $\mathbb{T}_yY \subset \PP^N$ used above
    is related to the intrinsic tangent space $TY$ by the exact sequence of bundles:
    \begin{align}\label{proj_tgt_space}
      0 \to \cO \to \mathbb{\hat T}Y(1) \to TY \to 0
    \end{align}
    where $\mathbb{\hat T}_yY \subset \C^{N+1}$
    is the vector subspace lying over $\mathbb{T}_yY \subset \PP^N$ and
    $\mathbb{\hat T}Y(1) = \mathbb{\hat T}Y \otimes \cO(1)$.
    We thus see that\footnote{Those
    comparing with \cite{MR630774} should note that the authors realize $\Theta$
    as $\PP(E^*)$ where $E = \mathbb{\hat T}^*Y(-1)$ (though they use Grothendieck's
    convention for projectivization so the dual on $E$ does not appear). It is slightly
    simpler to realize $\Theta$ as we do since then the tautological bundle
    $\cO_\Theta(-1)$ (which appears later) is not twisted.} $\Theta = \PP(\mathbb{\hat T}Y)$.

    Note that we have natural maps:
    \begin{align}\label{Theta_fibrations}
      \xymatrix{
          & \Theta \ar[dl]_f \ar[dr]^g & \\
        Y &                            & \Sec(Y)
      }
    \end{align}
    and that the fibre of $g$ above a point $z \in \Sec(Y) - Y$ is naturally identified
    by $f$ with the corresponding tangent locus in $Y$.

    With this setup in place, the proof is mostly formal.
    The result is a consequence of the relation that exists in $K(F)$ as a result
    of the generalized Euler sequence \eqref{gen_Euler} with $V = \mathbb{\hat T}Y$
    restricted to $F$ together with the fact that $F$ is a quadric.
    We thus consider the following exact sequence on $\Theta$:
    \begin{align}\label{four_term_exact_seq}
      0 \to \cO \to f^*\mathbb{\hat T}Y(1) \to T\Theta \to f^* TY \to 0
    \end{align}
    Furthermore there is a natural isomorphism
    $\cO_\Theta(1) \simeq g^*\cO(1)$ and so when we restrict \eqref{four_term_exact_seq}
    to an irreducible component $F$ of a fibre of the map $g$ we have:
    \begin{align}\label{tautological_bundle_F}
      f^*\mathbb{\hat T}Y(1)|_F \simeq \mathbb{\hat T}Y|_F
    \end{align}

    Now we simply collect up
    all the natural exact sequences to hand and interpret them as relations
    in $K(F)$ (forgetting the holomorphic structures).
    At the risk of being overly explicit, we list all the exact sequences we need below.
    We use the notation $\PP^{\delta}$ to denote
    the linear subspace of $\PP^N$ that is the span of the quadric $F$:
    \begin{equation*}
      \begin{CD}
        0 @>>> \cO_{\PP^{\delta}} @>>> \cO_{\PP^{\delta}}(1)^{\delta+1} @>>> T\PP^{\delta}       @>>> 0\\
        0 @>>> TF                 @>>> T\PP^{\delta}|_F                 @>>> \cO_F(2)            @>>> 0\\
        0 @>>> TF                 @>>> T\Theta|_F                       @>>> \cO_F^{2n+1-\delta} @>>> 0\\
      \end{CD}
    \end{equation*}

    Regarding these three sequences together with
    \eqref{proj_tgt_space} and \eqref{four_term_exact_seq}
    as five equations in $K(F)$ in five unknowns we can solve for the
    class of $\mathbb{\hat T}Y$. Bearing in mind
    \eqref{tautological_bundle_F} we get the following equation in $K(F)$:
    \begin{align*}
      \mathbb{\hat T}Y(1+\cO(1)) &= 2n+2 - \delta + (\delta + 1)\cO(1) - \cO(2)\\
                                 &= 2(n-\delta) + (2 + \delta - \cO(1))(1 + \cO(1))
    \end{align*}

    Thus, letting $W = \mathbb{\hat T}Y - 2 - \delta + \cO(1)$ we have:
    \begin{align}\label{division_equation}
      (1+\cO(1))W = 2(n-\delta)
    \end{align}
    which proves the result.
  \end{proof}

  We can already extract useful information from this proposition
  using characteristic classes. Taking the first Chern class of the identity
  \eqref{division_equation} we get:
  \begin{align*}
    2c_1(W) = -(n-\delta)c_1(\cO(1))
  \end{align*}
  Thus if $\dim F \ge 3$ since $c_1(\cO(1)) \in H^2(F, \Z) \simeq \Z$
  is a generator we must have $2\mid n-\delta$ as integers\footnote{In fact
  although $H^2(F, \Z)$ is not cyclic for $\dim_\C F = 2$ we can still deduce that
  $2\mid n-\delta$ in this case since $c_1(\cO(1))$ is not even and thus the
  relation holds as long as $\delta \ge 3$.}.
  However as we shall see a much stronger relationship holds.

  Fujita and Roberts \cite{MR630774} and Tango \cite{MR670136} essentially pursued
  this characteristic class approach (for Severi varieties) but only obtained partial results.
  To bring this approach to fruition it would be necessary to fully characterize the
  image of $K(F)$ under the Chern character, as a maximal-rank lattice in $H^*(F, \Q)$.
  In fact it is easier to dispense with ordinary cohomology entirely and stay in K-theory.

  Thus to take full advantage of the result of proposition \ref{LQEL_k-thy_divisibility}
  we need to know the ring structure of $K(F)$ explicitly. We have relegated a discussion
  of this purely topological result to proposition \ref{quadric_k_theory} in
  appendix \ref{kthy_quadric_appendix}. With this in hand we can state:
  \begin{corollary}\label{LQEL_divisibility_corollary}
    Let $Y \subseteq \PP^N$ be an $n$-dimensional LQEL manifold of type $\delta \ge 3$ then:
    \begin{align}\label{LQEL_divisibility}
      \left. 2^{\left[\frac{\delta-1}{2}\right]} ~ \right| ~ n - \delta \mbox{\quad in $\Z$}
    \end{align}
    In other words, we have a new proof of Russo's
    Divisibility Property for LQEL manifolds (see \cite{MR2501303} Theorem 2.8 (2))
    showing that it holds for topological reasons.
  \end{corollary}
  \begin{proof}[Proof:]
    This is an immediate consequence of proposition \ref{LQEL_k-thy_divisibility}
    together with corollary \ref{general_quadric_division_corollary}.
  \end{proof}

  \begin{remark}\label{refined_divis_rmk}
    In fact proposition \ref{LQEL_k-thy_divisibility}
    can be refined slightly: the class in $K(F)$ denoted $W$ in \eqref{division_equation}
    can be represented by the normal bundle of the entry locus
    (restricted to the tangent locus). Indeed if $F \subset Q \subset Y$ is the
    inclusion of a (general) tangent locus in an entry locus of $Y$ then we have
    the following natural exact sequences involving normal bundles:
    \begin{equation*}
      \begin{CD}
        0 @>>> \mathbb{\hat T}Y   @>>> \cO^{N+1}      @>>> N_{Y|\PP^N}(-1)        @>>> 0\\
        0 @>>> N_{F|Y}            @>>> N_{F|\PP^N}    @>>> N_{Y|\PP^N}            @>>> 0\\
        0 @>>> N_{F|\PP^{\delta}} @>>> N_{F|\PP^N}    @>>> N_{\PP^{\delta}|\PP^N} @>>> 0\\
        0 @>>> N_{F|Q}            @>>> N_{F|Y}        @>>> N_{Q|Y}                @>>> 0\\
      \end{CD}
    \end{equation*}
    and since $N_{F|Q} \simeq \cO(1)$, $N_{F|\PP^{\delta}} \simeq \cO(2)$,
    $N_{\PP^{\delta}|\PP^N} \simeq \cO(1)^{N - \delta}$ we get:
    \begin{align*}
      W = N_{Q|Y}(-1) \mbox{\quad in $K(F)$}
    \end{align*}

    In other words, we can refine proposition \ref{LQEL_k-thy_divisibility} to:
    \begin{align*}
      N_{Q|Y}\oplus N_{Q|Y}(-1) \mbox{\quad is topologically stably trivial restricted to $F$}
    \end{align*}
    Also, there is presumably a holomorphic counterpart of this statement, just as there
    is for the analogous statement \eqref{normal_bundle_iso_Kthy} discussed in the next
    section (though it is certainly not that the above holds as holomorphic bundles).
  \end{remark}

  Finally we wish to comment on Severi varieties. We thus recall:
  \begin{definition}
    A Severi variety is a non-degenerate non-singular irreducible
    variety $Y \subseteq \PP^N$ of dimension $n$ such that
    $3n = 2(N - 2)$ and $\Sec(Y) \ne \PP^N$.
  \end{definition}
  As we have noted, Zak \cite{MR1234494} provided a beautiful classification of Severi varieties
  showing that there are just four and that they correspond to projective planes
  over the four division algebras.
  The hard part of the classification is proving that $n \in \{2, 4, 8, 16\}$.
  
  The first step toward understanding Severi varieties is the following result
  of Zak:
  \begin{proposition}
    A Severi variety is an LQEL\footnote{In fact Zak's result is slightly stronger:
    a Severi variety is a QEL manifold (in the terminology of \cite{MR2501303}) i.e.,
    the entry loci are irreducible.} manifold
    of type $\delta = n/2$.
  \end{proposition}
  \begin{proof}[Proof:] See \cite{MR1234494} proposition 2.1 or \cite{MR2028046} proposition 3.2.3
  \end{proof}

  Our motivation for this work was the conjecture of Atiyah and Berndt (\cite{MR2039984},
  pp. 25,26) that there should be a K-theoretic proof of the dimension restriction
  for Severi varieties:
  \lq\lq \emph{There is a striking resemblance between Zak's theorem in complex algebraic
  geometry and the classical results about division algebras and projective planes.
  [...]
  One is therefore tempted to expect a K-theory proof of Zak's
  theorem}\rq\rq.

  For emphasis we thus explicitly state:
  \begin{corollary}
    Let $Y \subset \PP^N$ be an $n$-dimensional Severi variety, then $n  \in \{ 2, 4, 8, 16\}$.
  \end{corollary}
  \begin{proof}[Proof:]
    By definition $n$ is even and if $n > 4$ then by \eqref{LQEL_divisibility}
    with $\delta = n/2$
    we immediately find $4 \mid n$ and thence $2^{n/4} \mid n$ from which the result follows.
  \end{proof}

  We thus settle the conjecture affirmatively. Moreover, granting the purely topological
  result \ref{quadric_k_theory} describing the ring structure of the K-theory of
  the quadric, our methods provide an extremely short (and easy) proof that the dimension of a
  Severi variety must be as above.

  For the sake of completeness we provide the chronology of proofs of this result. It
  has been proved by:
  \begin{itemize}
    \item Zak (c.1982) \cite{MR1234494} (see also \cite{MR808175}) who used a detailed
          algebro-geometric study of the entry loci and their mutual intersection properties.
    \item Landsberg (1996) \cite{MR1422359} who studied the local differential geometry via
          the second fundamental form and appealed to classification of Clifford modules.
    \item Chaput (2002) \cite{MR1900320} who showed how to see a priori that a Severi
          variety is projectively homogeneous.
    \item Russo (2009) \cite{MR2501303} who established corollary \ref{LQEL_divisibility_corollary}
          by inductively studying the variety of lines through a point in an LQEL manifold.
    \item Schillewaert, Van Maldegham (2013) \cite{SchillewaertVanMaldegham}
         who show how to obtain the classification over arbitrary fields using only the
         axioms of what they call a Mazzocca-Melone set.
  \end{itemize}

  \begin{remark}
    We also remark that, as noted in \cite{MR1832903} \S 7, for a Severi variety
    the map $\Theta \to \Sec(Y)$ considered in the proof of proposition
    \ref{LQEL_k-thy_divisibility} is an example of a desingularization that Kempf
    \cite{MR0424841} calls collapsing a vector bundle.
  \end{remark}

\section{Dual varieties}\label{dual_vars_section}
  Proposition \ref{LQEL_k-thy_divisibility} is really just an examination of the consequences
  that exist in $K$-theory as a result of the relation obtained from the
  generalized Euler sequence on the bundle of embedded tangent spaces.

  However there is another bundle of embedded linear spaces associated to any non-singular
  variety, the (twisted) conormal bundle. I.e., if $N_{Y|\PP^N}$ is the normal bundle
  of a non-singular variety $Y \subseteq \PP^N$ and $y \in Y$ then there is a natural
  embedding of the fibre:
  \begin{align*}
    \PP(N_{Y|\PP^N}^*)_y \subseteq {\PP^N}^*
  \end{align*}
  It is thus natural to
  examine what consequences the generalized Euler sequence for the
  projectivized conormal bundle has in $K$-theory.

  Unsurprisingly, we will end up recovering known results (the Landman parity theorem
  and a weak version of a result due to Ein) but it is instructive to see the
  parallels with section \ref{LQEL_section} and to obtain these results with such ease.

  We thus define $\Phi = \PP(N_{Y|\PP^N}^*(1))$ and note that naturally
  $\Phi \subseteq Y\times Y^* \subseteq \PP^N\times {\PP^N}^*$ where $Y^*$ is the dual
  variety of $Y$. The analogue of the diagram
  \eqref{Theta_fibrations} in this case is then:
  \begin{align}\label{Phi_fibrations}
    \xymatrix{
        & \Phi \ar[dl]_f \ar[dr]^g & \\
      Y &                          & Y^*
    }
  \end{align}
  This time the fibre of $g$ above a general point $H \in Y^*$ is the contact locus
  $C_H(Y)$. Identifying this fibre with its image under $f$ we have:
  \begin{align*}
    C_H(Y) = \{y \in Y\mid \mathbb{T}_yY \subseteq H\}
  \end{align*}
  The contact locus is well known to be a linear space of dimension $k = N - 1 - \dim Y^*$,
  the dual deficiency of $Y$.
  Since we will obtain a relation in $K(C_H(Y))$ we must assume $k > 0$ in order to
  have non-trivial content.

  \begin{proposition}
    Let $Y \subset \PP^N$ be a irreducible non-singular variety of dual deficiency $k > 0$,
    let $H \in Y^*$ be a general point and let $N_{C|Y}$ be the normal bundle
    of the contact locus $C_H(Y)$ in $Y$ then we have:
    \begin{align*}
      N_{C|Y} = N_{C|Y}^*(1) \mbox{\quad in $K(C_H(Y))$}
    \end{align*}
  \end{proposition}
  \begin{proof}[Proof:]
    Referring to \eqref{Phi_fibrations}, we have the generalized Euler sequence
    for $\Phi$:
    \begin{align}\label{Phi_gen_Euler_seq}
      0 \to \cO \to f^*N_{Y|\PP^N}^*(1) \otimes g^*\cO(1) \to T\Phi \to f^*TY \to 0
    \end{align}
    where we have used $\cO_{\PP(N_{Y|\PP^N}^*(1))}(1) \simeq g^*\cO(1)$ naturally.

    Restricting to the fibre $C_H(Y)$ of $g$ as in the proof of proposition
    \ref{LQEL_k-thy_divisibility} and bearing in mind that $C_H(Y)$ is a linear space
    we thus have the following natural exact sequences:
    \begin{equation*}
      \begin{CD}
        0 @>>> \cO     @>>> \cO(1)^{k+1}    @>>> TC_H(Y)         @>>> 0\\
        0 @>>> TY      @>>> T\PP^N|_Y       @>>> N_{Y|\PP^N}     @>>> 0\\
        0 @>>> \cO     @>>> \cO(1)^{N+1}    @>>> T\PP^N          @>>> 0\\
        0 @>>> TC_H(Y) @>>> T\Phi|_{C_H(Y)} @>>> \cO^{N - 1 - k} @>>> 0\\
      \end{CD}
    \end{equation*}

    Regarding these four exact sequences together with \eqref{Phi_gen_Euler_seq} as relations
    in $K(C_H(Y))$ we thus obtain:
    \begin{align}\label{K_reln_Landman}
      N_{Y|\PP^N}^*(1) - N_{Y|\PP^N} = (k - N)(\cO(1) - 1)\mbox{\quad in $K(C_H(Y))$}
    \end{align}

    Since we are restricting to $C_H(Y) \subset Y$ we can instead express this in terms of
    the normal bundle $N_{C|Y}$ of $C_H(Y)$ in $Y$ instead of $N_{Y|\PP^N}$. These are related by the
    natural exact sequence of bundles on $C_H(Y)$:
    \begin{align*}
      0 \to N_{C|Y} \to N_{C|\PP^N} \to N_{Y|\PP^N}|_{C_H(Y)} \to 0
    \end{align*}
    and since $C_H(Y) \subset \PP^N$ is linearly embedded $N_{C|\PP^N} \simeq \cO(1)^{N-k}$.
    Thus eliminating $N_{Y|\PP^N}$ the identity \eqref{K_reln_Landman} becomes:
    \begin{align}\label{normal_bundle_iso_Kthy}
      N_{C|Y} = N_{C|Y}^*(1) \mbox{\quad in $K(C_H(Y))$}
    \end{align}
  \end{proof}

  \begin{corollary}
    Let $Y \subset \PP^N$ be an $n$-dimensional non-singular irreducible projective
    variety with dual deficiency $k > 0$ then:
    \begin{align*}
      2 \mid n-k
    \end{align*}
  \end{corollary}
  \begin{proof}[Proof:]
    Take first Chern classes of each side in \eqref{normal_bundle_iso_Kthy}. Since
    $\rank N_{C|Y} = n - k$, we get:
    \begin{align*}
      2c_1(N_{C|Y}) = (n-k)c_1(\cO(1))
    \end{align*}
    The result then follows since
    $c_1(\cO(1)) \in H^2(C_H(Y), \Z) \simeq \Z$ is a generator.
  \end{proof}

  The above corollary is known as Landman's parity theorem and was first proved by Landman
  using Picard-Lefshetz theory (though not published). Subsequently Ein
  \cite{MR853445} (using a result of Kleiman) provided a proof in which he established that
  \eqref{normal_bundle_iso_Kthy} in fact holds as holomorphic bundles
  rather than just as stable topological bundles as we have shown (see also \cite{MR2113135}
  theorem 7.1 and \cite{IonescuRussoDD} proposition 3.1).

  We note that in contrast to proposition \ref{LQEL_k-thy_divisibility}, the fact that there exists
  a bundle satisfying the identity
  \eqref{normal_bundle_iso_Kthy} in $K(C_H(Y))$ does not contain more information than we
  have obtained by noting that
  $c_1(N_{C|Y})$ is integral. For example the bundle
  $V = (1\oplus\cO(1))\otimes \cO^{(n-k)/2}$
  has rank $n-k$ and
  satisfies $V \simeq V^*(1)$ for any $n, k$ as long as $2 \mid n - k$.
  There is thus no analogue of the stronger corollary
  \ref{LQEL_divisibility_corollary} in this context.

  On the other hand, the fact that it is not just any bundle but $N_{C|Y}$ that
  appears in \eqref{normal_bundle_iso_Kthy} does of course contain more data. For example
  if $Y$ is a non-singular scroll of fibre dimension $l$ and base dimension $m < l$
  we can use it to calculate $k$.

  Indeed since the contact locus for a scroll is necessarily contained in fibre, i.e.,
  $C_H(Y) \subseteq L \subseteq Y$ for a fibre $L$, we have the natural
  exact sequence of normal bundles:
  \begin{align*}
    0 \to N_{C|L} \to N_{C|Y} \to N_{L|Y}|_{C_H(Y)} \to 0
  \end{align*}
  but of course $N_{C|L} \simeq \cO(1)^{l - k}$ and $N_{L|Y} \simeq \cO^m$
  and so in $K(C_H(Y))$ we have:
  \begin{align*}
    N_{C|Y} = m + (l-k)\cO(1) \mbox{\quad in $K(C_H(Y))$}
  \end{align*}
  The only way this is compatible with \eqref{normal_bundle_iso_Kthy} is if $k = l-m$.

\appendix

\section{K-theory of the quadric}\label{kthy_quadric_appendix}
  To take full advantage of proposition \ref{LQEL_k-thy_divisibility} we need to
  know the ring structure of the K-theory of a non-singular quadric. Surprisingly,
  this does not seem\footnote{We should qualify this remark by saying that since
  the $n$-dimensional complex quadric is diffeomorphic to the oriented real Grassmannian
  $\tilde G(2, n+2)$, it might be possible to extract the result we need 
  from \cite{MR1796336}. However as $\tilde G(2, n+2)$ is an edge case for the calculations in
  \cite{MR1796336}, it was difficult to be certain if it was really covered. Furthermore
  the polynomial ring representation of the K-theory given in \cite{MR1796336} is not
  perfectly suited to our needs. For these reasons
  and because we needed to be sure of the correctness of this crucial result,
  we decided to work from first principles.}
  to appear in the literature so we provide the necessary results here.

  The calculation falls into two cases depending on whether the dimension of
  the quadric is odd or even. As a CW complex, the quadric has a cell
  decomposition with no odd-dimensional cells and one cell in each even dimension
  except for the middle dimension in the case of the even-dimensional quadric where
  there are two cells.
  Thus\footnote{See e.g., \cite{MR0224083} proposition 2.5.2.} if $F$ is our quadric
  then $K^1(F)$ vanishes
  and $K^0(F) = K(F)$ is free-Abelian with rank equal to the number of cells, i.e.:
  \begin{align}\label{rank_K_F}
    \rank K(F) = \left\{
    \begin{array}{ll}
      1+\dim F & \mbox{$\dim F$ odd}\\
      2+\dim F & \mbox{$\dim F$ even}
    \end{array}\right.
  \end{align}

  To determine the ring structure of $K(F)$, we need to use more sophisticated techniques.
  We shall represent $F$ as a homogeneous space so that we can use the methods
  of Atiyah and Hirzebruch \cite{MR0139181} and Hodgkin \cite{MR0388371}.
  Thus let $\dim F = m-1$ and recall that there is a diffeomorphism:
  \begin{align*}
    F \simeq \frac{SO(m+1)}{SO(2)\times SO(m-1)}
  \end{align*}
  In fact we need $F$ to be a homogeneous space of a simply-connected group.
  Thus we lift to the double-cover and so regard:
  \begin{align}\label{quadric_spin_homogeneous}
    F \simeq \frac{Spin(m+1)}{Spin^c(m-1)}
  \end{align}
  (We need to be a little careful with the above for $m=2, 3$ but there is no real problem.)

  In view of \eqref{quadric_spin_homogeneous} we see that
  representations of $Spin^c(m-1)$ give vector bundles on $F$. We wish to highlight
  the bundles corresponding to certain special representations.

  Thus consider the double cover $Spin(2)\times Spin(m-1)$ of $Spin^c(m-1)$ and suppose for
  now that $m$ is even. If we let $\Z[t, t^{-1}]$ be the representation ring of $SO(2)$, then
  $RSpin(2) = \Z[t^{1/2}, t^{-1/2}]$. In addition there is the unique irreducible spin representation
  $\delta$ of $Spin(m-1)$ since $m-1$ is odd. Neither $t^{1/2}$ nor $\delta$
  descends to $Spin^c(m-1)$ but their product does. We thus let:
  \begin{align*}
    X = \mbox{bundle on $F$ obtained from representation $t^{-1/2}\delta$ of $Spin^c(m-1)$}
  \end{align*}
  Similarly for $m$ odd we define the bundles $X^+, X^-$ by:
  \begin{align*}
    X^\pm = \mbox{bundle on $F$ obtained from representation $t^{-1/2}\delta^\pm$ of $Spin^c(m-1)$}
  \end{align*}
  where $\delta^\pm$ are the irreducible components of the spin representation
  (since $m-1$ is even).
  Note that $\rank X = 2^{m/2-1}$ and $\rank X^\pm = 2^{(m-1)/2}$.

  \begin{proposition}\label{quadric_k_theory}
    Let $F \subset \PP^{m}$ be an $(m-1)$-dimensional non-singular quadric, $m \ge 3$. Let
    $L = \cO(1) - 1 \in K(F)$. Suppose $m$ is even and let $X$ be the bundle defined above, then:
    \begin{itemize}
      \item $1, L, L^2, \ldots L^{m-2}, X$ are a $\Z$-basis for the torsion-free ring $K(F)$
      \item $L^m = 0$ (obviously, for dimensional reasons)
      \item $LX = 2^{m/2} - 2X$
      \item $2^{m/2} X = 2^{m-1} - 2^{m-2}L + \cdots + 2L^{m-2} - L^{m-1}$ (this
            is equivalent to previous bullet but shows why we need $X$ instead of $L^{m-1}$)
    \end{itemize}
    (There is also a slightly-complicated formula for $X^2$ which we don't need so we
    suppress.)

    Similarly if $m$ is odd and $X^\pm$ are the bundles defined above, then:
    \begin{itemize}
      \item $1, L, L^2, \ldots, L^{m-2}, X^+, X^-$ are a $\Z$-basis for the torsion-free ring $K(F)$
      \item $L^m = 0$
      \item $LX^\pm = 2^{(m-1)/2} - X^\pm - X^\mp$
      \item $2^{(m-1)/2}(X^+ + X^-) = 2^{m-1} - 2^{m-2}L + \cdots - 2L^{m-2} + L^{m-1}$
    \end{itemize}
    (Again there are slightly-complicated formulae for $(X^\pm)^2$ and $X^+X^-$ which
    we suppress.)
  \end{proposition}
  \textbf{Proof} For brevity, let $G = Spin(m+1)$ and $H = Spin^c(m-1)$. We will use the methods 
  of Atiyah and Hirzebruch \cite{MR0139181} \S 5
  as well as Hodgkin \cite{MR0388371} to compute $K(G/H)$. Indeed as 
  pointed out by Atiyah and Hirzebruch, there is a natural map:
  \begin{align*}
    RH \to K(G/H)
  \end{align*}
  Now $H$ is a maximal-rank subgroup of $G$ and so $RG \subset RH$.
  The restriction to $RG$ gives only trivial
  bundles so if we let $RG$ act on $\Z$ by dimension then we have a natural
  map:
  \begin{align*}
    RH\otimes_{RG}\Z \to K(G/H)
  \end{align*}
  Hodgkin (\cite{MR0388371} page 71)
  proves this map is an isomorphism since $\pi_1(G) = 1$ and $H$ has maximal rank. Furthermore
  there is a natural exact sequence of $RH$-modules:
  \begin{align*}
    0 \to RH\cdot I \to RH \to RH\otimes_{RG}\Z \to 0
  \end{align*}
  where $I \subset RG \subset RH$ is the augmentation ideal of $RG$
  (i.e., the kernel of the dimension map $RG \to \Z$).
  In other words for general reasons we have a natural ring isomorphism:
  \begin{align}\label{K_thy_homogeneous_space}
    K(F) \simeq RH / RH\cdot I
  \end{align}
  To put this to use we need an explicit realization of three things:
  \begin{itemize}
    \item $RG$ and the dimension map $RG \to \Z$ with kernel $I$
    \item $RH$
    \item The inclusion $RG \hookrightarrow RH$
  \end{itemize}
  We must now separately consider the two cases $m$ even and $m$ odd. We consider first
  the slightly-simpler case $m$ even.

  We shall follow the notation of Husemoller \cite{MR1249482}; by Theorem 10.3 op. cit.
  we have that $RG$ is a polynomial ring:
  \begin{align*}
    RG &\simeq \Z[\Lambda_1, \Lambda_2, \ldots, \Lambda_{m/2-1}, \Delta]
  \end{align*}
  and $\Delta^2 = 1 + \Lambda_1^2 + \cdots + \Lambda_{m/2-1}^2 + \Lambda_{m/2}^2$.

  Now $H = (Spin(2)\times Spin(m-1))/\{\pm 1\}$ and so we have:
  \begin{align*}
    RH \simeq (RSpin(2)\otimes RSpin(m-1))^{\Z/(2)}
  \end{align*}
  If we let\footnote{We need to be a little careful for the case $m=4$ below but there
  is no real problem. However the statement clearly does not hold for $m=2$; hence the
  assumption $m \ge 3$ in the proposition statement.}:
  \begin{align*}
    RSpin(2)    &= \Z[t^{1/2}, t^{-1/2}]\\
    RSpin(m-1) &= \Z[\lambda_1, \ldots, \lambda_{m/2-2}, \delta]
  \end{align*}
  then as above $\delta^2 = 1 + \lambda_1^2 + \cdots + \lambda_{m/2-1}^2$.
  The $\Z/(2)$ action fixes the $\lambda_i$ and changes the sign of $\delta$ as well
  as the half-integral powers of $t$. We thus obtain:
  \begin{align}\label{RH}
    RH \simeq \Z[t, t^{-1}, \lambda_1, \ldots, \lambda_{m/2-1}, X]
  \end{align}
  where $X = t^{-1/2}\delta$.
  Note that the above ring is not quite a polynomial ring, it is a quotient
  by the ideal generated by the relation:
  \begin{align}\label{X_k_reln}
    X^2 = t^{-1}(1 + \lambda_1^2 + \cdots + \lambda_{m/2-1}^2)
  \end{align}

  Finally the map $RG \hookrightarrow RH$ is described by:
  \begin{align*}
    \Delta = (t^{1/2} + t^{-1/2})\delta = (1+t)X
  \end{align*}
  and
  \begin{align*}
    \Lambda_i = \lambda_i + (t + t^{-1})\lambda_{i-1} + \lambda_{i-2}
  \end{align*}
  for $1 \le i \le m/2-1$
  provided we agree that $\lambda_0 = 1$ and $\lambda_{-1} = 0$.
  By \eqref{K_thy_homogeneous_space} we thus have the following relations
  between the images of elements of $RH$ in $K(G/H)$:
  \begin{align}
    (1+t)X &= \dim\Delta = 2^{m/2}\label{spin_k_reln}\\
    \lambda_i + (t + t^{-1})\lambda_{i-1} + \lambda_{i-2} &= \dim\Lambda_i\label{lambda_k_reln}
  \end{align}

  Using \eqref{lambda_k_reln} inductively we remove the $\lambda_i$ from any polynomial expression
  in $RH$ given by \eqref{RH}
  and have only expressions involving $t, t^{-1}$ instead. In other words we
  thus have a surjection from $\Z[t, t^{-1}, X]$ to $K(F)$.

  Now it is easier to
  work with nilpotent elements so let $L = t-1$. Note that
  $t$ corresponds to $\cO(1)$ so this is indeed the $L$ in the proposition statement. Then
  $L^{m} = 0$ for dimensional reasons (its image under Chern character would
  lie in cohomology of degree at least $2m$ and $\dim_\R F = 2m-2$) and so we have:
  \begin{align*}
    t^{-1} = 1 - L + L^2 - \cdots - L^{m-1}
  \end{align*}
  We thus have a surjection $\Z[L, X]$ to $K(F)$. Combining this with the
  relation \eqref{X_k_reln} we see that $K(F)$ is spanned over $\Z$
  by the classes represented by: $L^i, XL^i$ for $0 \le i \le m-1$.
  From here using \eqref{spin_k_reln} we see that $K(F)$ is spanned by: $L^i, X$
  for $0 \le i \le m-1$ and then finally elementary computation reveals:
  \begin{align*}
    2^{m/2} X = 2^{m-1} - 2^{m-2}L + \cdots + 2L^{m-2} - L^{m-1}
  \end{align*}
  Thus we can omit $L^{m-1}$ and still have a spanning set. Since there are
  $m$ elements in this set and we know by \eqref{rank_K_F} that
  the rank of $K(F)$ is $m$, this must be a $\Z$-basis
  as required. This deals with the case $m$ even.
  
  The argument for the case $m$ odd is extremely similar.
  For the methods below we need to assume $m \ge 5$ but the result for
  the case $m=3$ is easily verified since in this case $F \simeq S^2\times S^2$.

  This time we have:
  \begin{align*}
    RG &\simeq \Z[\Lambda_1, \ldots, \Lambda_{(m-3)/2}, \Delta^+, \Delta^-]\\
    RSpin(m-1) &\simeq \Z[\lambda_1, \ldots, \lambda_{(m-5)/2}, \delta^+, \delta^-]
  \end{align*}
  and:
  \begin{align*}
    (\Delta^\pm)^2 &= \Lambda_\pm + \Lambda_{(m-3)/2} + \Lambda_{(m-7)/2} + \cdots\\
    \Delta^+\Delta^- &= \Lambda_{(m-1)/2} + \Lambda_{(m-5)/2} + \cdots
  \end{align*}
  where $\Lambda_{(m+1)/2} = \Lambda_+ + \Lambda_-$ and
  the series end in $1$ or $\Lambda_1$ according to parity (and similarly for $\delta^\pm$ and
  $\lambda_\pm$). Then similarly to the case
  $m$ even we have:
  \begin{align*}
    RH \simeq \Z[t, t^{-1}, \lambda_1, \ldots, \lambda_{(m-5)/2}, \lambda_+, \lambda_-, X^+, X^-]
  \end{align*}
  where $X^\pm = t^{-1/2}\delta^\pm$ and the map $RG \hookrightarrow RH$ is given
  by the same relation between the $\lambda_i$ and $\Lambda_i$ as for $m$ even but:
  \begin{align*}
    \Delta^+ &= t^{1/2}\delta^+ + t^{-1/2}\delta^-\\
    \Delta^- &= t^{1/2}\delta^- + t^{-1/2}\delta^+
  \end{align*}
  Using these formulae, the same argument goes through just as for $m$ even to yield
  the stated results. \qed

  \begin{corollary}\label{general_quadric_division_corollary}
    Let $F \subset \PP^m$ be a non-singular quadric hypersurface, $m \ge 3$,
    and suppose $1 + \cO(1)$ divides
    $l$ in $K(F)$ for some $l \in \Z$ then:
    \begin{align*}
      \left. 2^{\left[\frac{m+1}{2}\right]} ~\right|~ l \mbox{\quad in $\Z$}
    \end{align*}
    (The brackets in the power denote the integer part.)
  \end{corollary}
  \begin{proof}[Proof:]
    Set $L = \cO(1) - 1$ as in the proposition. If $m$ is even, let
    $X$ be as in the proposition and if $m$ is odd, let $X = X^+ + X^-$. Note that
    in either case we then have:
    \begin{align*}
      (1 + \cO(1))X = (2+L)X = 2^{\left[\frac{m+1}{2}\right]}
    \end{align*}
    Since $2+L$ is not a zero divisor and $X$ is part of a $\Z$-basis
    the result follows.
  \end{proof}

\bibliography{paper}
\bibliographystyle{plain}
\end{document}